\documentclass{amsart}
\pdfoutput=1 
\usepackage{amssymb, latexsym}
\usepackage{graphicx}
\usepackage{hyperref}

\title{Virtual Geometricity is Rare}
\author{Christopher H. Cashen} 
\address{
Fakult\"at f\"ur Mathematik\\
Universit\"at Wien\\
Oskar-Morgenstern-Platz 1\\
1090 Vienna \\Austria}
\email{\href{mailto:christopher.cashen@univie.ac.at}{christopher.cashen@univie.ac.at}}

\author{Jason F. Manning}
\address{
Department of Mathematics\\
310 Malott Hall\\
Cornell University\\
Ithaca, NY  14853\\
USA
}
\email{\href{mailto:jfmanning@math.cornell.edu}{jfmanning@math.cornell.edu}}
\date{March 16, 2015}

\hypersetup{
    pdftitle={Virtual Geometricity is Rare},    
    pdfauthor={Christopher H. Cashen and Jason F. Manning},     
    pdfkeywords={geometric word, virtually geometric, random word, free group}, 
    colorlinks=true,       
    linkcolor=black,          
    citecolor=black,        
    filecolor=black,      
    urlcolor=black           
}

\theoremstyle{plain}
\newtheorem*{reftheorem}{Theorem}
\newtheorem{theorem}{Theorem}[section]
\newtheorem{lemma}{Lemma}[section]
\newtheorem{proposition}{Proposition}[section]
\newtheorem{corollary}{Corollary}[section]

\newtheorem*{question}{Question}

\theoremstyle{remark}
\newtheorem*{claim}{Claim}

\newtheorem*{remark}{Remark}

\theoremstyle{definition}
\newtheorem{definition}{Definition}[section]

\newenvironment{claimproof}[1][Proof of Claim]{\vspace{1ex}\noindent{\it #1.}\hspace{0.5em}}
	{\hfill$\lozenge$\vspace{1ex}}

\def\makeautorefname#1#2{\expandafter\def\csname#1autorefname\endcsname{#2}}
\let\fullref\autoref

\makeautorefname{theorem}{Theorem} 
\makeautorefname{lemma}{Lemma} 
\makeautorefname{proposition}{Proposition} 
\makeautorefname{corollary}{Corollary} 
\makeautorefname{definition}{Definition}
\makeautorefname{example}{Example}
\makeautorefname{section}{Section}
\makeautorefname{subsection}{Section}
\makeautorefname{subsubsection}{Section}
\makeautorefname{claim}{Claim}

\makeatletter 
\let\c@lemma=\c@theorem 
\makeatother
\makeatletter 
\let\c@proposition=\c@theorem 
\makeatother
\makeatletter 
\let\c@corollary=\c@theorem 
\makeatother
\makeatletter 
\let\c@definition=\c@theorem 
\makeatother
\makeatletter 
\let\c@example=\c@theorem 
\makeatother

\makeatletter
\newsavebox\myboxA
\newsavebox\myboxB
\newlength\mylenA

\newcommand*\xoverline[2][0.75]{%
    \sbox{\myboxA}{$\m@th#2$}%
    \setbox\myboxB\null
    \ht\myboxB=\ht\myboxA%
    \dp\myboxB=\dp\myboxA%
    \wd\myboxB=#1\wd\myboxA
    \sbox\myboxB{$\m@th\overline{\copy\myboxB}$}
    \setlength\mylenA{\the\wd\myboxA}
    \addtolength\mylenA{-\the\wd\myboxB}%
    \ifdim\wd\myboxB<\wd\myboxA%
       \rlap{\hskip 0.5\mylenA\usebox\myboxB}{\usebox\myboxA}%
    \else
        \hskip -0.5\mylenA\rlap{\usebox\myboxA}{\hskip 0.5\mylenA\usebox\myboxB}%
    \fi}
\makeatother
\newcommand{\inv}[1]{\xoverline{#1}}

\DeclareMathOperator{\Aut}{Aut} 
\newcommand{\bdry}{\partial} 
\newcommand{\R}{\mathbb{R}}
\newcommand{\F}{\mathbb{F}}
\newcommand{\C}{\mathbb{C}}
\newcommand{\bZ}{\mathbb{Z}}
\newcommand{\closure}[1]{\overline{#1}} 
\newcommand{\from}{\colon\thinspace}
\newcommand{\co}{\from}

\newcommand{\mot}{w} 
\newcommand{\multimot}{\underline{w}} 
\newcommand{\multicon}{[\underline{w}]}
\newcommand{\mott}{v} 
\newcommand{\basis}{\underline{x}}
\newcommand{\basic}{x}

\DeclareMathOperator{\Wh}{\mathfrak{W}} 

\newcommand{\X}{\mathcal{X}} 

\newcommand{\tree}{\mathcal{T}} 
\newcommand{\lines}{\textbf{L}} 
\renewcommand{\line}{\mathcal{L}} 

\newcommand{\linesw}{\lines_{\multimot}}

\newcommand{\therank}{r}

\begin{document}
\begin{abstract}
We present the results of computer experiments suggesting that 
the probability that a random multiword in a free group
is virtually geometric decays to zero exponentially quickly in the
length of the multiword.
We also prove this fact.

\end{abstract}

\maketitle


\section{Introduction}
Let $\F$ be a finite rank non-abelian free group. 
Fix, once and for all, a
basis $\basis=\{\basic_1,\dots,\basic_\therank\}$.
A \emph{word in $\F$} is an element of $\F$ expressed as a freely
reduced word in the letters $\basis^\pm$.
A \emph{multiword} $\multimot=\{w_1,\ldots,w_k\}$ is a finite subset of $\F$.
The set $\multimot$ determines a collection of conjugacy classes $\multicon=\{[w_1],\ldots,[w_k]\}$ in $\F$, possibly with multiplicity.

\begin{definition}\label{def:geometric}
  Fix an identification of $\F$ with the fundamental group of an orientable  $3$--dimensional handlebody $H$. 
  The set of conjugacy classes $\multicon$ determines a free homotopy class of map 
  $\sqcup_{w\in \multimot}S^1\to H$.  
The multiword $\multimot$ is
   (orientably) \emph{geometric} if the  
   homotopy class determined by $\multicon$ contains an embedding into $\bdry H$.

Similarly, a multiword $\multimot$ is \emph{non-orientably geometric} there is such an embedding where we allow
$H$ to be a non-orientable handlebody.
\end{definition}

\begin{remark}
Geometricity does not depend on the choice of identification
$\F=\pi_1(H)$, since the handlebody group, that is, the group of
(orientation preserving) homeomorphisms of $H$ modulo isotopy, surjects onto the outer
automorphism group of $\pi_1(H)$.
\end{remark}

\begin{definition}
  If $[w]$ is a conjugacy class in $\F$ and $F<\F$ is a finite index
  subgroup, we can ``lift'' $[w]$ to $F$ as follows.  
Let $[w]_F$ be the set of $F$--conjugacy classes of the form $g^{-1}w^\alpha g$, where $g\in \F$ and  $\alpha=\alpha(w,g)\geq 1$ is minimal subject to the requirement that $g^{-1}w^\alpha g \in F$.  The lift $\multicon_F$ of $\multicon$ to $F$ is then defined to be $\sqcup_{w\in \multimot} [w]_F$.

(From a topological point of view, let $\tilde{H}\to H$ be the cover corresponding to $F<\F$.
A conjugacy class $[w]$ corresponds to the free homotopy class of some map $\phi\co S^1\to H$, and $[w]_F$ 
corresponds to the collection of free homotopy classes of elevations of $\phi$ to $\tilde{H}$.)
\end{definition}

\begin{definition}
  A multiword $\multimot$ is \emph{virtually geometric} if
  there exists a finite index subgroup $F$ of $\F$ such that $\multicon_F$ is geometric.
\end{definition}

The Baumslag-Solitar words $\inv{b}a^pb\inv{a}^q$ in $\F_2=\langle
a,b\rangle$ with $p\neq 0\neq q$ and $|p|\neq |q|$ are examples of
words that are virtually geometric but not geometric \cite[Section 6]{GorWil10}.

If $\multimot$ is orientably geometric then $\langle
\basis\mid\multimot\rangle$ is the fundamental group of a compact
3--manifold. This can be seen by taking the handlebody $H$ with
$\pi_1(H)=\F$ and attaching a 2--handle along each component of a regular neighborhood of the
embedded multicurve representing $\multicon$.

If $\multimot$ is virtually geometric then there exists a positive
number $n$ such that $\langle \basis\mid \mot^n \text{ for
}\mot\in\multimot\rangle$ is virtually a 3--manifold group (see \cite[Remark~1.5]{Man10}).

Gordon and Wilton \cite{GorWil10} studied virtual geometricity as an
approach to a special case of Gromov's surface subgroup conjecture.
They show that if $\multimot$ is virtually geometric, then the
\emph{double of $\F$ along $\multimot$}, constructed by taking two
copies of $\F$ and identifying the two copies of each
$\mot\in\multimot$ via cyclic amalgamation, is virtually a 3--manifold
group, and contains a surface subgroup when it is 1--ended.
In an early version of  \cite{GorWil10}
(\cite[Question~22]{GorWil09}), they ask whether any non-virtually geometric word exists.  One reason to suspect words to
be always virtually geometric is the effectiveness in low-dimensional
topology of arguments ``desingularizing'' immersed submanifolds in
finite-sheeted covers.  Two important examples of such arguments are
Papakyriakopoulos' ``Tower argument'' for the Loop and Sphere Theorems
\cite{Pap57} and Scott's argument for subgroup separability in surface
groups \cite{Sco78}.  However, there do exist non-virtually geometric
words;  the second author gave a (non-generic) criterion in \cite{Man10} and
exhibited some words which satisfy it.

One might still wonder how common virtual geometricity is.
We wrote a computer program that determines if a given multiword is
virtually geometric or not, and set it to work testing random
multiwords in low rank free groups.
Our experiments, presented in \fullref{sec:experiment}, suggest that the probability that a random multiword
is virtually geometric decays to zero exponentially quickly in the
length of the multiword.
We also make explicit estimates for the rate of exponential decay.
Surprisingly, our experiments suggest that the ratio of the number of virtually
geometric words to the number of geometric words of a given length is
bounded above.

\begin{question}
  Does the ratio of virtually geometric words to geometric words stay bounded as the length goes to infinity?  Does it tend to $1$?
\end{question}

In Sections \ref{sec:whitehead}--\ref{sec:vgrare} of the paper we apply the technology developed in \cite{CasMac11,Cas10splitting} to establish the result suggested by
the experiments.  In the terminology of Section \ref{sec:generic}, we show:
\begin{reftheorem}
Virtual geometricity is exponentially rare.  
\end{reftheorem}
More precisely, recall that $\F$ is a free group of fixed finite rank $r>1$.
In \fullref{main} we show that the 
proportion of words of length $l$ in $\F$ which are virtually geometric decays to zero exponentially quickly in $l$.
The same is true if we restrict our attention to the subset $\C\subseteq \F$ consisting of cyclically reduced words.

For multiwords there are different models of genericity. 
However, a multiword that contains a non-virtually geometric word is
itself non-virtually geometric, so virtual geometricity will also be
exponentially rare for any model in which a random multiword contains
a long random word. 
In particular, virtual geometricity is exponentially rare for
multiwords in both the ``few relators model'' of genericity
(\fullref{corollary:fewrelators}) and the ``density model'' (\fullref{corollary:density}). 

The rough idea of the proof is to find a ``poison'' word $v\in \F$
which obstructs the virtual geometricity of any cyclically reduced
$w\in \F$ containing $v$ as a subword.  

The first author \cite{Cas10splitting} characterized virtually
geometric multiwords as those that are constructed as amalgams of
geometric pieces. In particular, if $\F$ does not admit cyclic
splittings relative to $\multimot$, that is, splittings in which the elements of $\multimot$ are elliptic, then
virtual geometricity reduces to geometricity. 
The poison word $v$ is a concatenation of words $v_1$ and $v_2$ so
that $v_1$ obstructs the existence of a relative splitting, and $v_2$
obstructs non-orientable geometricity.  
The characterization from \cite{Cas10splitting} implies that $v=v_1v_2$ obstructs virtual geometricity.

Finally we appeal to the well-known fact (\fullref{prop:growthsensitive}) that cyclically reduced words exponentially generically contain every short word -- in particular they contain $v$.

In fact there is the slight complication that our word $v$ is only poisonous to Whitehead minimal words, but these are exponentially generic by a result of Kapovich--Schupp--Shpilrain \cite{MR2221020}.

\subsection{Acknowledgments}
Thanks to Nathan Dunfield for pointing out \fullref{lem:fewpowers},
to John Mackay for explaining the note in
\fullref{prop:growthsensitive}, and to the referee for helpful
comments and the reference
to Solie's paper \cite{Sol12}.

The first author was supported by the Agence Nationale de la Recherche
(ANR) grant ANR-2010-BLAN-116-01 GGAA, the European Research Council
(ERC) grant of Goulnara ARZHANTSEVA, grant agreement \#259527, and the ICERM Program ``Low-dimensional Topology, Geometry, and Dynamics".
The second author was partly supported by the National Science
Foundation, grants DMS 1104703 and 1343229.
This project also benefited greatly from a visit funded by the GEAR network, grants DMS 1107452, 1107263, 1107367.

\section{Generic Sets}\label{sec:generic}

Our definitions in this section follow Kapovich, Schupp, and Shpilrain \cite{MR2221020}.
A sequence $(c_n)\subset\R$ with $\lim_{n\to\infty}c_n=c\in\R$
\emph{converges exponentially fast} if there exist $a>0$ and
$b\in\R$ such that $|c-c_n|\leq \exp(b-an)$ for all sufficiently
large $n$.

Let $|\mot|$ denote the length of $\mot$ in the word metric on $\F$
corresponding to $\basis$.

\begin{definition}
Let $A\subset B\subset\F$.
  The set $A$ is \emph{generic} in $B$ if:
\[\lim_{n\to\infty}\frac{\#\{\mot\in A\mid |\mot|\leq
  n\}}{\#\{\mot\in B\mid |\mot|\leq n\}}=1\]

$A$ is \emph{exponentially generic} in $B$ if the convergence is
exponentially fast.

A subset is \emph{rare}, or \emph{negligible}, if the complement is
generic.
It is \emph{exponentially rare}, or \emph{exponentially negligible},
if the complement is exponentially generic. 
\end{definition}

It is an easy computation to see that the intersection of finitely
many (exponentially) generic sets is (exponentially) generic.

A property $\mathcal{P}$ is said to be (exponentially) generic/rare in
$B$ if the set of words having $\mathcal{P}$ is (exponentially) generic/rare.

Let $\C$ be the set of cyclically reduced words in $\F$ with respect to the basis $\basis$.

\begin{theorem}[{\cite[Theorem
  ~B(1)]{MR2221020}}]\label{corollary:genericmultiword}
The set of cyclically reduced words that are not proper powers and are Whitehead minimal is exponentially generic in $\C$.
The set of words that are not proper powers and whose cyclic reduction is Whitehead
minimal is exponentially generic in $\F$.
\end{theorem}

\begin{proposition}\label{prop:growthsensitive}
Let $\mot$ be a word in $\F$.
The subset of words that contain $\mot$ as a subword is exponentially
generic in $\F$. The subset of cyclically reduced words that contain
$\mot$ as a subword is exponentially generic in $\C$.
\end{proposition}
\begin{proof}
This fact is well known. For $\F$, see \cite[Section~2]{GriDeL97} or
  \cite[Corollary~4.4.9]{LinMar95}. For $\C$, see
  \cite[Lemma~2.5]{Mac12}.  
Note that the statement of the latter result includes the assumption
that $|\mot|>4$, but in fact the estimate and proof given there are
valid for any subword $\mot$ once the random word length is greater
than $16$.
\end{proof}

It is also true that for a fixed word $\mot\in\F$ the set of words
that contain $\mot$ as a subword of their cyclic reduction is
exponentially generic. 
This can be deduced from \cite[Proposition~6.2]{MR2221020}, which is
stated without proof.
For completeness, we give a proof in
\fullref{prop:cyclicreductiongrowthsensitive}.
The proof uses a few auxiliary lemmas.

\begin{lemma}\label{lemma:rareinspheresimpliesrareinballs}
Let $A\subset \F$.  Let $S_n$ be the set of words in $\F$ of length exactly $n$, and suppose there exist $a>0$, $b\in \R$ so that
\[\frac{\#A\cap S_n}{\#S_n}\leq \exp(b -an)\]
for all sufficiently large $n$.  Then $A$ is exponentially rare in $\F$.
\end{lemma}
\begin{proof}
  In fact one can show that for any $0<a'<\min\{a,\ln(2r-1)\}$, there is a $b'$ so that $\frac{\#A\cap B_n}{\#B_n}\leq \exp(b'-a'n)$.  This straightforward computation is left to the reader.
\end{proof}

Let $\lceil x\rceil$ denote the least integer greater than or equal 
to $x$. 
\begin{definition}
  Define the \emph{middle third} of a word $\mot\in\F$ of length
  $l\geq 5$ to be the
  subword of $\mot$
obtained by discarding the first and last $\lceil \frac{l}{3}\rceil$ letters. 
\end{definition}

\begin{lemma}\label{lemma:middlethridgeneric}
  Let $\mot$ be a word in $\F$. The set of words containing $\mot$ as
  a subword of their middle third is exponentially generic in $\F$.
\end{lemma}
\begin{proof}
The ratio of the number of words of length $l$ not containing $\mot$
as a subword of their middle third to the number of words of length $l$
is equal to the ratio of the number of words of length $m=l-2\lceil\frac{l}{3}\rceil$ not containing $\mot$ as
a subword to the number of words of length $m$. 
By \fullref{prop:growthsensitive}, this ratio decays exponentially in
$m$, but $m$ is linear in $l$, so the ratio decays exponentially in $l$.
Conclude by applying \fullref{lemma:rareinspheresimpliesrareinballs}.
\end{proof}

\begin{lemma}\label{lemma:middlethirdsurvives}
  The set of words $\mot\in\F$ such that the middle third of $\mot$
  is a subword of the cyclic reduction of $\mot$ is exponentially
  generic in $\F$.
\end{lemma}
\begin{proof}
Let $\therank$ be the rank of $\F$.
Let $n(l)$ be the number of words of length $l$ for which cyclic reduction reduces
  length by at least $2\lceil\frac{l}{3}\rceil$. 
Every such word is of the form $\mot \mott \mot^{-1}$, where $\mot$ is
of length $\lceil\frac{l}{3}\rceil$, $\mott$ is of length
$m=l-2\lceil\frac{l}{3}\rceil$, and the last letter of $\mot$ is not the
inverse of the first letter of $\mott$.
Thus, $n(l)\leq 2\therank(2\therank-1)^{(m-1)}\cdot
(2\therank-1)^{\lceil\frac{l}{3}\rceil}$.
In fact, the inequality is strict because for some choices there
will be a free reduction in $\mott \mot^{-1}$, resulting in a word of
length less than $l$. We have:
\[\frac{n(l)}{2\therank(2\therank-1)^{(l-1)}}<\frac{2\therank(2\therank-1)^{(m-1)}\cdot
(2\therank-1)^{\lceil\frac{l}{3}\rceil}}{2\therank(2\therank-1)^{(l-1)}}\leq\frac{1}{(\sqrt[3]{2\therank-1})^l}\]
Now, \fullref{lemma:rareinspheresimpliesrareinballs} implies
 the set of words $\mot$ for which the middle third of $\mot$ survives cyclic
reduction of $\mot$ is exponentially generic in $\F$.
\end{proof}

\begin{proposition}\label{prop:cyclicreductiongrowthsensitive}
  Let $\mot$ be a word in $\F$. The set of words that contain $\mot$
  as a subword of their cyclic reduction is exponentially generic in $\F$.
\end{proposition}
\begin{proof}
The set of words containing $\mot$ as a subword of their cyclic
reduction contains the intersection of the set of words containing $\mot$ as a
 subword of their middle third with the set of words whose middle
 third survives cyclic reduction. 
By \fullref{lemma:middlethridgeneric} and
\fullref{lemma:middlethirdsurvives}, both of these sets are
exponentially generic in $\F$, so their intersection is as well.
\end{proof}

\begin{definition}
Let $B\subset\F$. 
  A word $\mot\in\F$ is \emph{poison to property} $\mathcal{P}$ in $B$ if no word 
  of $B$ containing $\mot$ as a subword of its cyclic reduction enjoys $\mathcal{P}$. 

A word $\mot$ is \emph{(exponentially) generically poison to}
$\mathcal{P}$ in $B$ if there exists a (exponentially) generic subset 
$A\subset B$ such that $w$ is poison to $\mathcal{P}$ in $A$. 
\end{definition}

A consequence of \fullref{prop:growthsensitive} and \fullref{prop:cyclicreductiongrowthsensitive} is:
\begin{corollary}\label{corollary:poison}
  If there exists a word that is (exponentially) generically poison to 
  $\mathcal{P}$ in $\F$ or $\C$ then 
  $\mathcal{P}$ is (exponentially) rare in $\F$ or $\C$, respectively. 
\end{corollary}

\section{Whitehead Graphs}\label{sec:whitehead}

\begin{definition}
A multiword  $\multimot=\{\mot_1,\dots,\mot_k\}$ such that each
$\mot_i$ is not a proper power, and such that $\mot_i$ is not
conjugate to $\mot_j$ or $\inv{\mot}_j$ for all $i\neq j$, is called \emph{unramified}.
\end{definition}

A multiword is called cyclically reduced if all of its elements are
cyclically reduced with respect to the fixed basis $\basis$ of $\F$.

Let $\tree$ be the Cayley graph of $\F$ with respect to $\basis$,
which is a $2|\basis|$--valent tree.
Let $\closure{\tree}=\tree\cup\bdry\tree$ denote the compactification of $\tree$ by its
Gromov boundary $\bdry\tree$.

\begin{definition}
  If $\multimot$ is a cyclically reduced multiword,
  define $\linesw$ to be the collection of distinct bi-infinite geodesics
  $[f\inv{\mot}^\infty,f\mot^\infty]\subset\closure{\tree}$ where
  $\mot\in\multimot$ and $f\in\F$.
\end{definition}

\begin{definition}\label{def:whiteheadgraph}
Let $\multimot$ be a cyclically reduced multiword.
Let $\X$ be a connected subset of $\tree$.
Let $\closure{\X}$ be
its closure in $\closure{\tree}$.
  The \emph{Whitehead graph of $\multimot$ over $\X$}, denoted
  $\Wh(\X)$
  is a graph whose vertices are in bijection
  with connected components of $\closure{\tree}\setminus\closure{\X}$.
Distinct vertices are joined by an edge for each $\line\in\linesw$ with
endpoints in the corresponding complementary components of $\closure{\X}$.
\end{definition}

\begin{remark}
  The Whitehead graph depends on $\multimot$ via $\linesw$. We
  suppress $\multimot$ from the notation, as it will always be clear
  from context.
\end{remark}

The following easy lemma clarifies the definition:
\begin{lemma}
  Let $C_1$ and $C_2$ be components of $\tree\setminus \X$.  Then
  $C_1$ and $C_2$ are connected by an edge in $\Wh(\X)$ if and only if the label of the shortest path joining them is a subword of $\mot^{\infty}$ or $\mot^{-\infty}$ for some $\mot\in \multimot$.
\end{lemma}

\begin{remark}
If $\multimot$ is unramified and cyclically reduced, and if $\X=*$ is
a single vertex, then the vertices of $\Wh(*)$ are in bijection with $\basis^\pm$, and there is
one edge from vertex $\inv{x}$ to vertex $y$ for every occurrence of
the subword $xy$ in $\multimot$, with words of $\multimot$ treated as
cyclic words.
This is the classical definition of the Whitehead graph. 
\end{remark}

Let $|\mot|$ denote the word length of $\mot$ with respect to the basis $\basis$ of $\F$.
Let $|[\mot]|$ denote the minimal word length of an element of the conjugacy class of $\multimot$.

\begin{definition}
An unramified, cyclically reduced multiword $\multimot=\{\mot_1,\dots,\mot_k\}$ is \emph{Whitehead minimal} if for every automorphism $\alpha\in\Aut(\F)$ we have:
\[\sum_{i=1}^k|\mot_i|\leq\sum_{i=1}^k|[\alpha(\mot_i)]|\]
\end{definition}

\section{Relative Splittings}\label{sec:split}

\begin{definition}
  A \emph{splitting of $\F$ relative to $\multimot$} is a splitting of
  $\F$ as a graph of groups such that each $\mot\in\multimot$ is elliptic. 
\end{definition}

The following lemma is essentially due to Whitehead \cite{Whi36}. See
also \cite{Ota92,Sta99,Sto97}.
\begin{lemma}\label{lem:freesplitting}
If $\Wh(*)$ is connected and has no cut vertices then
$\F$ does not split freely relative to $\multimot$.
\end{lemma}

The next lemma is a consequence of \cite[Lemma~4.9]{CasMac11}:
\begin{lemma}\label{lem:cyclicsplitting}
If $\F$ splits over $\langle\mott\rangle$ relative to $\multimot$ 
  then $\Wh([\inv{\mott}^\infty,\mott^\infty])$
  has more than one connected component.
\end{lemma}

\section{Filling Words and Full Words}\label{sec:full}

Kapovich and Lustig \cite{KapLus10} define a non-trivial element $\mot\in\F$ to be
\emph{filling} if it has non-zero translation length for every very
small isometric action of $\F$ on an $\R$--tree.  (An action of $\F$ on an $\R$--tree is \emph{very small} if tripod stabilizers are trivial and arc stabilizers are maximal cyclic.)
Work of Guirardel \cite{Gui98} shows that $\mot$ is filling if and only
if $F$ does not split freely or cyclically relative to $\mot$.
Kapovich and Lustig ask for a combinatorial criterion that implies a
word is filling. 
The first such criterion was given by Solie \cite{Sol12}. 
We will give another.
Solie uses an exponentially generic set constructed by Kapovich,
Schupp, and Shpilrain \cite{MR2221020}.
Essentially, the set consists of words that are
 balanced, in the sense that every
element of $\basis^\pm$ occurs roughly the same number of times in
$\mot$ and every reduced two letter word in $\basis^\pm$ occurs as a subword of
$\mot$ roughly the same number of times (see \cite[Proof of
Theorem~A]{MR2221020}). Solie shows these words are filling. 
Our condition essentially says that a word is filling if it is
sufficiently complicated, in the sense that every reduced three letter
word in $\basis^\pm$ occurs as a subword of $\mot$.
Both Solie's condition and ours are satisfied on exponentially
generic sets in $\F$; ours is somewhat simpler to check.

If $\mot$ and $\mott$ are words in $\basis^\pm$, we say $\mot$
\emph{cyclically contains} $\mott$ if the free reduction of $\mott$
appears as a subword of the cyclic reduction of a power of $\mot$. 
We say a multiword $\multimot$ cyclically contains $\mott$ if one of
the words of $\multimot$ cyclically contains $\mott$.

\begin{definition}\label{def:full}
  A multiword $\multimot$ is \emph{full} if for every reduced
  word $\mott$ in $(\basis^\pm)^3$ either $\mott$ or $\inv{\mott}$ is cyclically
  contained in $\multimot$.
\end{definition}

\begin{lemma}\label{lem:fullisgeneric}
  The set of full words is exponentially generic in $\F$.
  The set of full, cyclically reduced words is exponentially generic in $\C$.
\end{lemma}
\begin{proof}
Let $B$ be either $\F$ or $\C$. 
Fix a reduced word $\mot$ containing every reduced word of $(\basis^\pm)^3$ as a 
subword.
The set $A$ of words in $B$ whose cyclic reduction contains $\mot$ is
exponentially generic in $B$, by \fullref{prop:growthsensitive} if
$B=\C$ or \fullref{prop:cyclicreductiongrowthsensitive} if $B=\F$.
\end{proof}

The rest of this subsection is devoted to establishing that $\F$ cannot split freely or cyclically relative to a full word.  
\begin{lemma}\label{lem:nofreesplitting}
  $\F$ does not split freely relative to a full word.
\end{lemma}
\begin{proof}
  Let $\mot$ be a full word.
  Its Whitehead graph contains the complete graph, so it is connected
  without cut vertices. 
  By \fullref{lem:freesplitting}, $\F$ does not split freely relative to
$\mot$.  
\end{proof}

\begin{proposition}\label{prop:line}
  If $\gamma\subset \tree$ is a line and $\mot$ is full, then $\Wh(\gamma)$ is connected.
\end{proposition}
\begin{proof}
  For $K\subset \gamma$ compact, let $C^\pm$ be the components of
  $\tree\setminus K$ containing the rest of $\gamma$.  
Let $V_K$ be the vertex set of $\Wh(K)$.
  Define $W(K)$ to be the full subgraph of $\Wh(K)$ on vertices $V_K\setminus\{C^\pm\}$.
  
  Notice that for $K\subseteq K'$ we have $W(K)\subseteq W(K')$, and moreover $\Wh(\gamma) = \varinjlim W(K)$.  
  The following claim thus suffices to establish the proposition.
  \begin{claim}
    $W(K)$ is connected, for any nonempty compact subsegment $K\subset \gamma$.
  \end{claim}
  \begin{claimproof}
    We may suppose that $\gamma$ is parametrized to have unit speed, so that $\gamma$ sends integers to vertices of $\tree$.  
    For $z\in\bZ$, let $s_z\in \basis^\pm$ be the label of the edge $\gamma|_{[z,z+1]}$.

    The segment $K$ is equal to $\gamma|_{[p,q]}$ for some integers $p, q$.  
    Each vertex $C$ of $W(K)$ is a component of $\tree\setminus K$, and there is a unique edge $e_C$ of $\tree$ starting on $K$ and ending in $C$.  We define $n_C\in\bZ$, $s_C\in \basis^\pm$ so that the initial point of $e_C$ is $\gamma(n_C)$, and the label of $e_C$ is $s_C$.
    The pair $(n_C,s_C)$ completely determines $C$, so we can also
    refer to the vertices of $W(K)$ by these pairs.  
    Namely, $(n,s)\in \bZ\times \basis^\pm$ is a vertex of $W(K)$ if
    and only if:
\[ p\leq n \leq q\mbox{ and } s\notin\{\inv{s}_{n-1},s_n\}. \]
    Two vertices $(n,s)$ and $(n,t)$ of $W(K)$ are connected by an
    edge if and only if one of $\inv{s}t$ or $\inv{t}s$ is a subword of $w^{\infty}$.
    Since $\mot$ is full, $(n,s)$ and $(n,t)$ are indeed
    connected by an edge. 
    Two vertices $(n,s)$ and $(n+1,t)$ are connected if and only if
    $\inv{s} s_n t$ or its inverse occurs in some $\mot^\infty$.
    Again, since $\mot$ is full, all of these edges occur.  
    It follows that $W(K)$ is connected.
  \end{claimproof}
\end{proof}

\begin{corollary}\label{cor:fullisrigid}
  $\F$ does not split freely or cyclically relative to a full word.
\end{corollary}
\begin{proof}
  Let $\mot$ be a full word.
  Lemma \ref{lem:nofreesplitting} tells us there is no free splitting.

  Suppose that $\F$ splits over $\langle v\rangle$ relative to
  $\mot$.  Then by \fullref{lem:cyclicsplitting} the
  Whitehead graph $\Wh([\bar{v}^\infty,v^\infty])$ has more than one connected component.  But this contradicts \fullref{prop:line}.
\end{proof}

In particular, fullness gives an easily verifiable, combinatorial
condition implying that a word is filling, giving another answer to
Kapovich and Lustig's question:
\begin{corollary}
  A full word in $\F$ is filling.
\end{corollary}


\section{Virtual Geometricity is Rare}\label{sec:vgrare}
\begin{definition}
  For $\basic_i\in \basis$ and $\mot \in \F$, an \emph{$\basic_i$--syllable} is a maximal subword of $\mot$ equal to a power of $\basic_i$ or $\inv{\basic}_i$.  A subword is a \emph{syllable} if it is a $\basic_i$--syllable for some $i$.
\end{definition}

\begin{lemma}[{\cite[Lemma p. 18]{Berge}}]\label{lem:fewpowers}
  Let $\multimot$ be unramified, cyclically reduced, Whitehead minimal
  and either geometric or nonorientably geometric.
Suppose that $\basic_1$ and $\inv{\basic}_1$ do not form a separating pair of
vertices in $\Wh(*)$, and suppose that no $\mot\in\multimot$  begins and ends with distinct $\basic_1$--syllables.
Up to absolute value, at most three different powers of $\basic_1$ appear as syllables of elements of $\multimot$.
\end{lemma}
\begin{proof}[Proof idea]
  Using a theorem of Zieschang \cite[Theorem p. 11]{Berge} one can show that if there were four we would be able to find four non-intersecting parallelism classes of properly embedded arcs in a punctured torus or punctured Klein bottle.  An Euler characteristic argument shows that there are at most three such classes.  
\end{proof}
\begin{remark}
  Berge's \cite{Berge} results are stated for orientable handlebodies, but the
  proofs of \cite[Lemma p. 18]{Berge} and Zieschang's theorem
  \cite[Theorem p. 11]{Berge} do not require orientability.
\end{remark}

\begin{theorem}\label{main}
Let $B$ be either $\F$
or $\C$.
Let $\mathrm{VG}(l)$ be the probability that that a word chosen
randomly with uniform probability from the set of words of $B$ of
length at most $l$ is virtually geometric.
There exist $a>0$ and $b\in\mathbb{R}$ such that $\mathrm{VG}(l)\leq\exp(b-al)$ for all sufficiently large $l$.
\end{theorem}
\begin{proof}
Let $\mott'$ be a reduced word with first letter $\basic_1$
and last letter $\basic_2$ that contains every word
in $(\basis^\pm)^3$.
Let $\mott=\basic_1^1\basic_2\basic_1^2\basic_2\basic_1^3\basic_2\basic_1^4\basic_2\mott'$.

Let $C$ be the subset of $B$ consisting
of words that are not proper powers and whose cyclic reduction is Whitehead
minimal. 
$C$ is exponentially generic in $B$ by \fullref{corollary:genericmultiword}.

Let $\mot'$ be a word in $C$ whose cyclic reduction $\mot$
contains
$\mott$ as a subword.
Then $\mot$ is full, so, by \fullref{cor:fullisrigid}, $\F$ does
not admit free or cyclic splittings relative to $\mot$.
In particular, the JSJ decomposition of $\F$ relative
  to $\mot$ is trivial, so \cite{Cas10splitting} implies $\mot$ is virtually geometric if
and only if it is either orientably or non-orientably geometric.
However, $\Wh(*)$ contains the complete subgraph on its
vertices, so it has no separating pairs of vertices.
By \fullref{lem:fewpowers}, $\mot$ can not be geometric or
non-orientably geometric, since at least four distinct
powers of $\basic_1$ appear as syllables of $\mot$.
Therefore, $\mot$ and $\mot'$ are not virtually geometric.

We have shown that $\mott$ is exponentially generically poison to
virtual geometricity in $B$, so
the theorem follows from \fullref{corollary:poison}.
\end{proof}

\fullref{main} is stated for single words.
For multiwords, there are different models of genericity, but
the presence of a single non-virtually geometric
word in a multiword implies that the multiword is non-virtually
geometric. 
Thus, virtual geometricity will be rare for any 
model in which a random multiword contains a long random word.
We state corollaries for two popular models.

\begin{corollary}[{`Few Relators Model' of \cite{MR1445193}}]\label{corollary:fewrelators}
  Let $B$ be either $\F$
or $\C$. 
For any $k\geq 1$, let $\mathrm{VG}(l)$ be the probability that a multiword consisting of
  $k$ randomly chosen words in $B$ of length at most $l$, selected
  independently and uniformly, is virtually
  geometric.
 There exist $a>0$ and $b\in\mathbb{R}$ such that 
$\mathrm{VG}(l)\leq\exp(b-al)$ for all sufficiently
 large $l$.
\end{corollary}

\begin{corollary}[{`Density Model' of \cite{Gro93}}]\label{corollary:density}
Let $B$ be either $\F$
or $\C$.
  For any density $0\leq d\leq 1$, let $\mathrm{VG}(l)$ be the probability that a multiword consisting of
  $(2\cdot\mathrm{rank}(\F)-1)^{dl}$ randomly chosen words in $B$ of length at most $l$, selected
  independently and uniformly, is virtually
  geometric. 
There exist $a>0$ and $b\in\mathbb{R}$ such that 
$\mathrm{VG}(l)\leq\exp(b-al)$ for all sufficiently
 large $l$.
\end{corollary}

Finally, we extract from the proof of \fullref{main} a non-virtual
geometricity criterion:
\begin{corollary}
  Let $\multimot$ be an unramified, Whitehead minimal, cyclically
  reduced multiword. If $\multimot$ is full and at least four distinct powers (up to absolute value) of some basis element occur as syllables of elements of $\multimot$, then
  $\multimot$ is not virtually geometric.
\end{corollary}

\label{sec:main}
\section{Experimental Estimates}\label{sec:experiment}
We wrote some computer scripts \cite{CasMan13} to determine if a given
multiword is virtually geometric or not.
The underlying theory was developed in \cite{Man10,
 CasMac11, Cas10splitting}.
We used these computer scripts to run computer experiments testing random
words for virtual geometricity. The results are presented in this section.
Let us first give a brief account of how the scripts work.

Given a multiword $\multimot$ and a word $\mott$, there is a way to
determine if $\F$ splits over $\langle \mott\rangle$ relative to
$\multimot$ by considering connectivity of certain finite generalized
Whitehead graphs. (We improve \fullref{lem:cyclicsplitting} to only
consider a bounded subsegment of $[\mott^{-\infty},\mott^\infty]$.)

More specifically, we implement the algorithm of \cite[Theorem~4.17]{CasMac11} to search
for splitting words. 
There are two ideas behind this search algorithm. 
First, there is a bound, depending on $\multimot$, of the maximal
length of a cyclically reduced word $\mott$ such that the splitting of $\F$ over $\langle
\mott\rangle$ relative to $\multimot$ is universal, ie, such that
every other splitting word is elliptic with respect to this
splitting. 

Second, for a cyclically reduced splitting word, the generalized
Whitehead graph over every prefix $p$ has a cut pair of a particular type, and this cut pair gives directions for finding the next letter of the splitting word.
Thus, we can search inductively starting with short words $\mott$ and checking
their generalized Whitehead graphs for these special cut pairs.  If we find one, take the extensions $p\basic$ for $\basic\in \basis^\pm$ suggested by the cut pairs.

The worst case estimates for the length of such a search are
horrendous, but on generic multiwords it can be done effectively in
low rank, because we quickly see that most short words can not be a
prefix of a splitting word. 

If $\F$ splits over $\langle \mott\rangle$, we deduce whether a second word $\mott'$ is hyperbolic or
elliptic in the splitting over $\langle \mott\rangle$ by adding edges
corresponding to $\mott'$ to the Whitehead graph over $\mott$ and
checking if the number of connected components stays the same or decreases.
We do this for all pairs of words found by the search algorithm to
find the universal splitting words. 

Given the list of conjugacy classes of universal splitting words, we compute the JSJ
decomposition of $\F$ relative to $\multimot$, again using
combinatorics of generalized Whitehead graphs. 
The main result of \cite{Cas10splitting} says that $\multimot$ is
virtually geometric if and only if the induced multiword in each
vertex of the relative JSJ decomposition is orientably or
non-orientably geometric. 
The induced multiword in a vertex group consists of conjugates of elements of
$\multimot$ contained in that vertex group plus the image of a
generator of each incident edge group.

The induced multiwords for quadratically hanging vertices of the
relative JSJ decomposition are always geometric, so virtual
geometricity of $\multimot$ is reduced to checking geometricity of
the induced multiword in each rigid vertex of the relative JSJ
decomposition. 
For this we use Berge's program \texttt{heegaard} \cite{Berge}.
Berge shows that geometricity is equivalent to the existence of a
planar embedding of the Whitehead graph of the multiword that
satisfies an additional consistency requirement.
The fact that we start with rigid vertices implies that such a
Whitehead graph is 3--connected, so there is a unique planar
embedding, if one exists. 
\texttt{heegaard} checks whether there exists a planar embedding, and,
if so, finds it and checks the consistency conditions. 

There is a slight complication that \texttt{heegaard} only checks the
consistency conditions for orientable geometricity.
To work around this issue, if \texttt{heegaard} says an induced
multiword in a rigid vertex group is not geometric, we also check the
lifts to all index 2 subgroups. 

In the figures below we present findings of our computer
experiments\footnote{IPython \cite{ipython} was used in the development
  of the computer scripts and in running the experiments. Scripts for testing virtual geometricity and reproducing our experiments can be found at \url{https://bitbucket.org/christopher_cashen/virtuallygeometric}.} 
on the proportions of random words which are geometric, virtually
geometric, and not full in ranks 2, 3, and 4. 
Experiments on geometricity and virtual geometricity were performed before we found the proof of \fullref{main}; the experiments on fullness were inspired by the proof.

We see in \fullref{fig:gvgfull} that while
the proportion of not full words provides an exponentially decaying
upper bound for the proportion of virtually geometric words, it is not
 very sharp.

Fit curves are computed for each data series by taking the subseries
that comes after the first word length for which the proportion of
words falls below 50\%, taking logarithms, computing a best fit line
by weighted least squares approximation,
and then exponentiating.

\begin{figure}[h!]
  \centering
  \includegraphics[width=\textwidth]{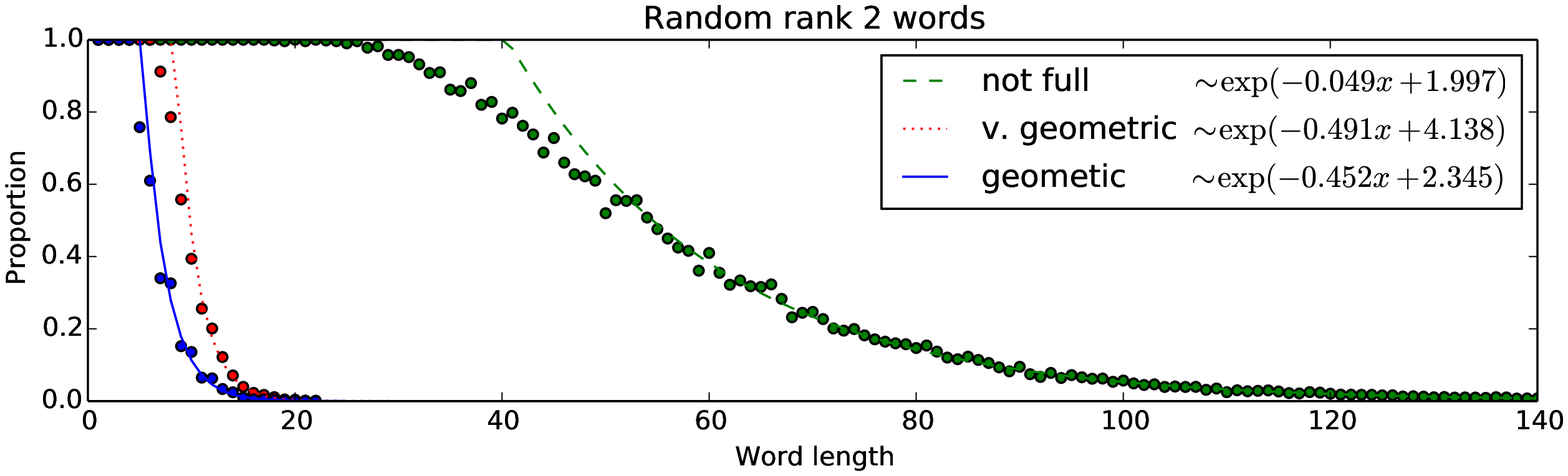}
  \includegraphics[width=\textwidth]{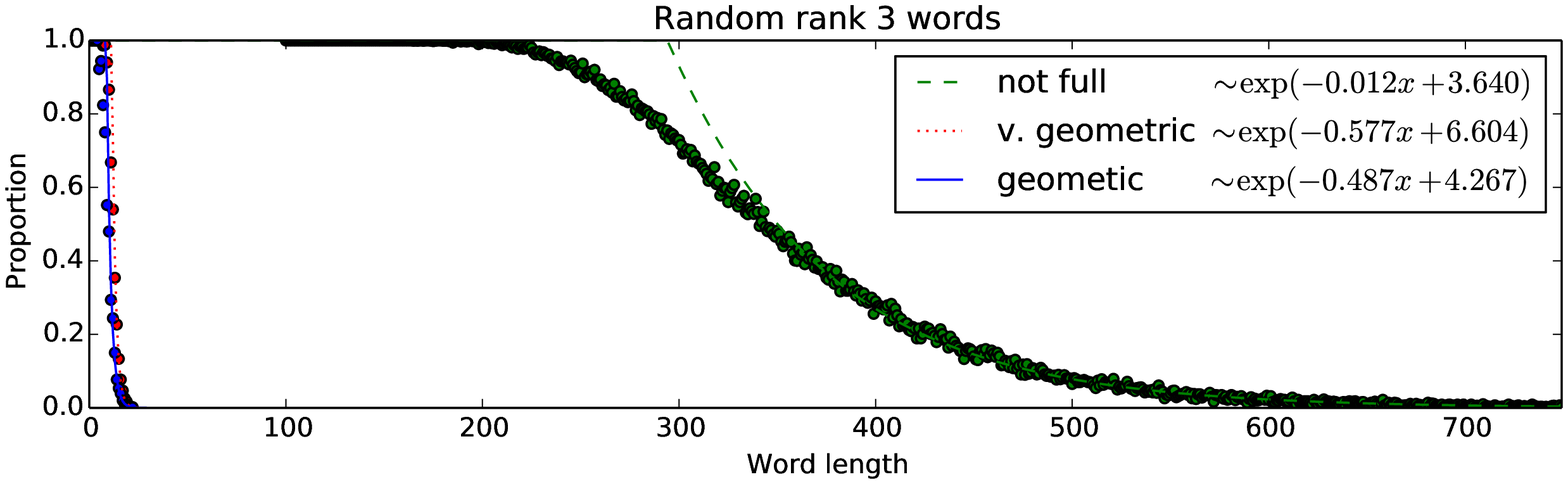}
  \includegraphics[width=\textwidth]{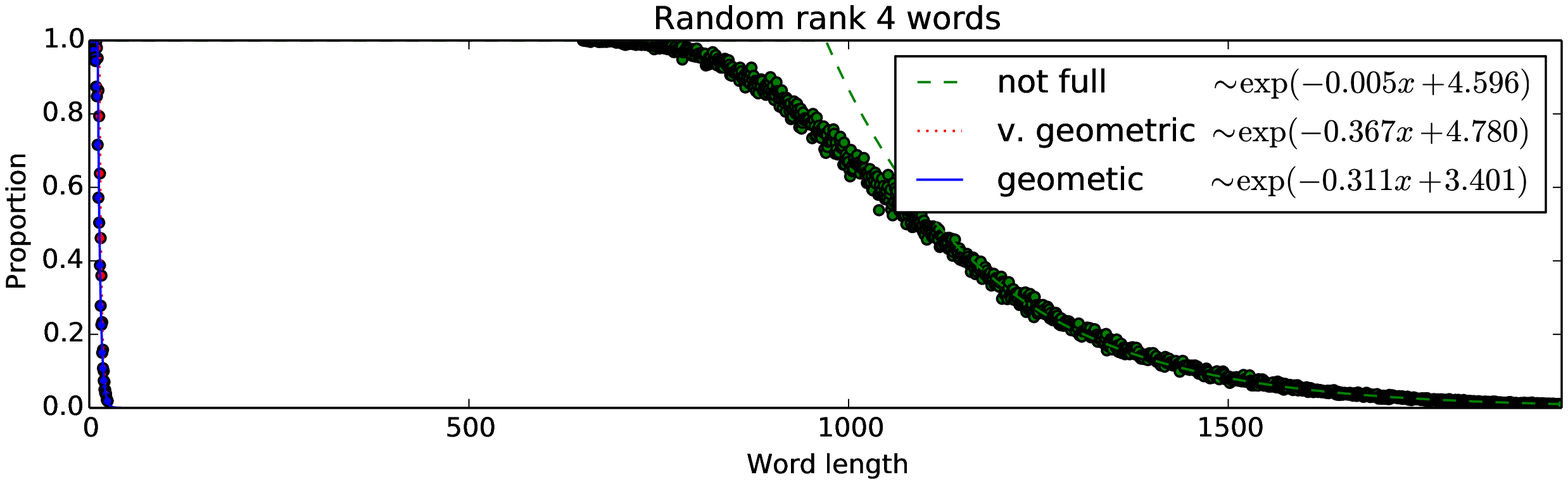}
  \caption{Geometricity, Virtual Geometricity, and Full Words}
  \label{fig:gvgfull}
\end{figure}

\fullref{fig:loggvg} plots logarithm(proportion $\pm$ standard error) and omits the full words data.
The number of trials increases with the word length so that the quantity
\[\log\left(\frac{\text{proportion} + \text{standard error}}{\text{proportion} - \text{standard
error}}\right)\] 
stays small.
Generating graphs with this amount of precision took about two months
of continuous running time on a circa 2010 dual core desktop computer.

\begin{figure}[h!]
  \centering
  \includegraphics[width=\textwidth]{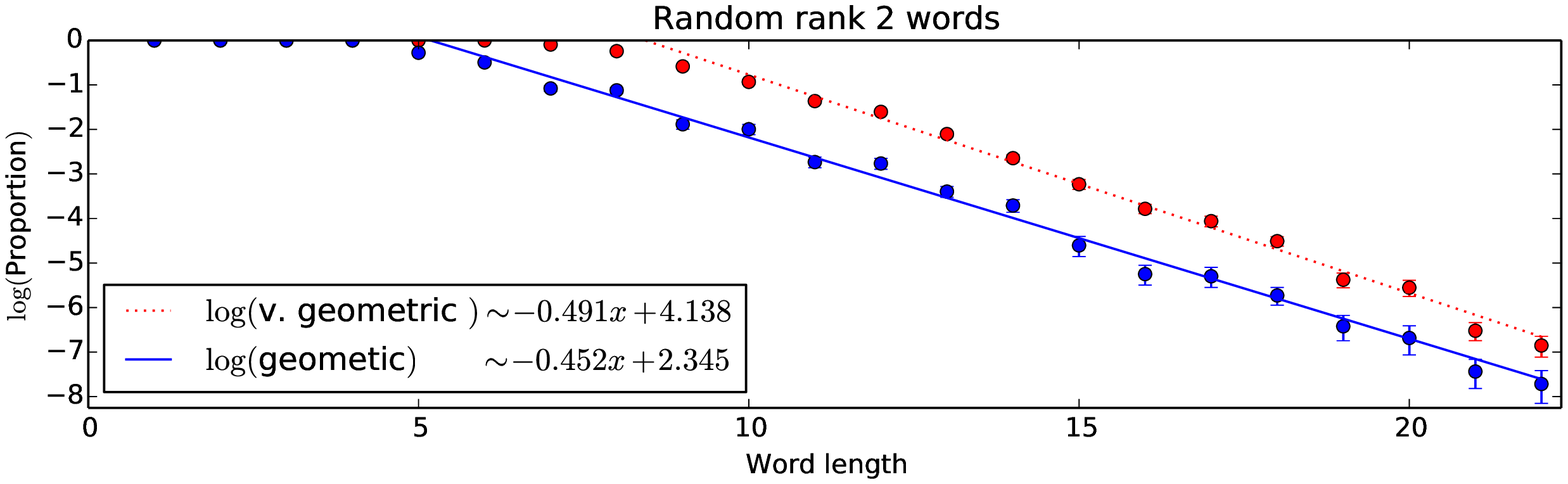}
  \includegraphics[width=\textwidth]{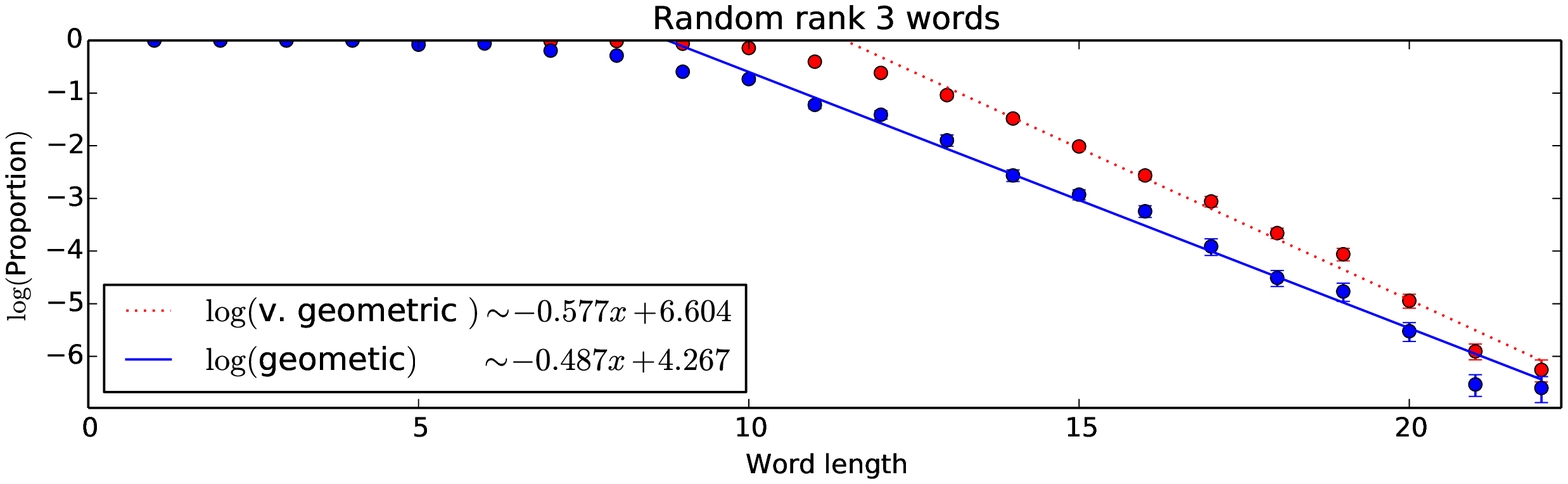}
  \includegraphics[width=\textwidth]{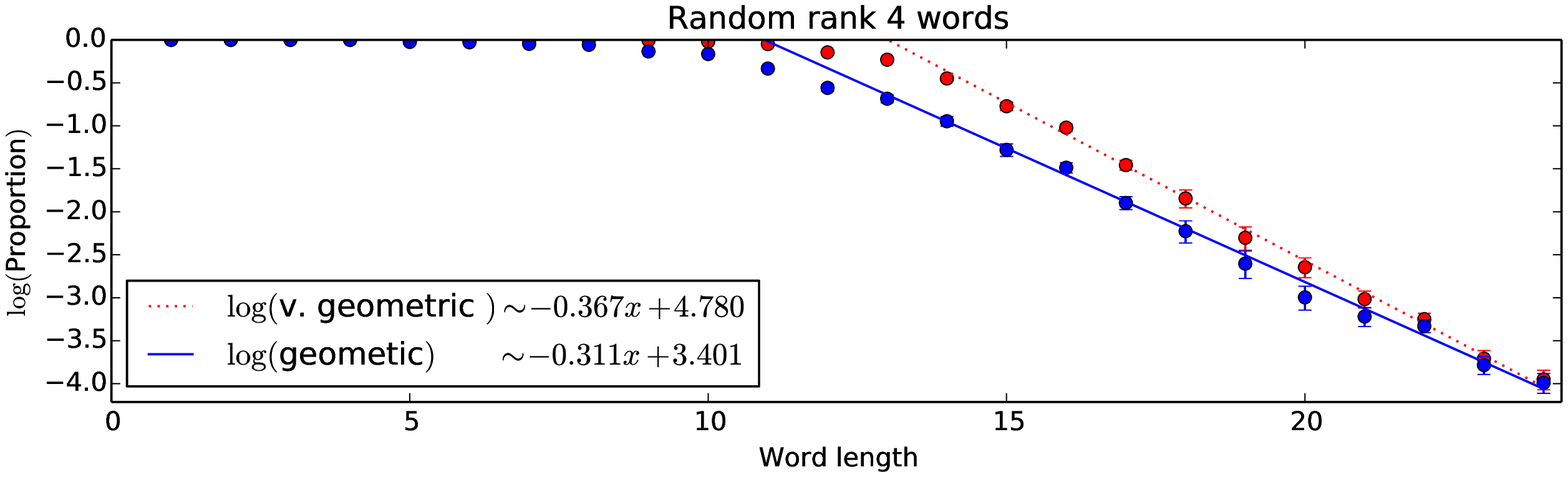}
  \caption{Geometricity and Virtual Geometricity}
  \label{fig:loggvg}
\end{figure}



\bibliographystyle{hyperamsplain}
\bibliography{vg}

\end{document}